\documentclass[12pt]{amsart}

\usepackage{latexsym}
\usepackage{psfrag}
\usepackage{amsmath}
\usepackage{amssymb}
\usepackage{epsfig}
\usepackage{amsfonts}
\usepackage{amscd}
\usepackage{mathrsfs}
\usepackage{graphicx}
\usepackage{enumerate}

\newtheorem{theorem}{Theorem}
\newtheorem{lemma}[theorem]{Lemma}

\newtheorem{cor}[theorem]{Corollary}
\newtheorem{prop}[theorem]{Proposition}

\newtheorem{defn}[theorem]{Definition}

\begin{document}

\title[Bressoud Style Identities]{Bressoud Style Identities for Regular Partitions \\ and Overpartitions}

\author[Kur\c{s}ung\"{o}z]{Ka\u{g}an Kur\c{s}ung\"{o}z}
\address{Faculty of Engineering and Natural Sciences, Sabanc{\i} University, \.{I}stanbul, Turkey}
\email{kursungoz@sabanciuniv.edu}



\date{September 2014}

\begin{abstract}
\noindent
We construct a family of partition identities which contain the following identities: 
Rogers-Ramanujan-Gordon identities, 
Bressoud's even moduli generalization of them, 
and their counterparts for overpartitions due to Lovejoy et al. and Chen et al.  
We obtain unusual companion identities to known theorems as well as to the new ones 
in the process.  
The proof is, against tradition, constructive and open to automation.  
\end{abstract}

\maketitle


\section{Introduction}
\label{secIntro}

Euler defined the integer partitions, 
noticed their intimate connection with $q-$series, 
and gave the first known identities in the field \cite[Ch. 16]{Euler}.  
Afterwards, arguably the longest stride was made by Rogers, Ramanujan, and Schur
by independently discovering the Rogers-Ramanujan identities \cite{Rogers, Hardy, Schur}.  
The first of the said identities is given below.  

\begin{theorem}[the first Rogers-Ramanujan identity]
\label{thRR1}
  Given a non-negative integer $n$, 
  the number of partitions of $n$ into distinct non-consecutive parts
  equals the number of partitions of $n$ into parts that are $\equiv \pm 1 \pmod{5}$.   
\end{theorem}

Then, one of the significant generalizations was given by Gordon \cite{RRG}.  

\begin{theorem}[Rogers-Ramanujan-Gordon identities]
\label{thRRG}
  Let $k$ and $a$ be integers such that $k \geq 2$ and $1 \leq a \leq k$.  
  Given a non-negative integer $n$, 
  the number of partitions of $n$ in which 
  the combined number of occurrences of any two consecutive parts is at most $k-1$
  and 1 can appear at most $a-1$ times 
  equals the number of partitions of $n$ into parts 
  that are $\not \equiv 0, \pm a \pmod{2k+1}$.  
\end{theorem}

The Rogers-Ramanujan-Gordon identities spawned fruitful research in the area.  
One of the follow-up questions was whether or not there was an even moduli analogue.  
This question was first addressed and partially solved by Andrews \cite{Andrews-RRG-evenmod-partial}, 
and finally settled by Bressoud \cite{Br-RRG-AllMod}.  
The extension case of his theorem was as follows, 
where $f_i$ denotes the number of occurrences of $i$ in a given partition.  

\begin{theorem}
\label{thBressoud}
  Given integers $k$, $a$ such that $k \geq 2$ and $1 \leq a < k$, 
  let $C_{k, a}(n)$ denote the number of partitions of $n$ such that 
  $f_1 < a$, $f_i + f_{i+1} < k$ for all $i$, 
  and if $f_i + f_{i+1} = k-1$, then $if_i + (i+1)f_{i+1} \equiv a-1 \pmod{2}$; 
  and $D_{k, a}(n)$ denote the number of partitions of $n$ into parts 
  $\not\equiv 0, \pm a \pmod{2k}$.  
  Then $C_{k, a}(n) = D_{k, a}(n)$.  
\end{theorem}

Recently, Corteel and Lovejoy defined overpartitions \cite{Ctl-Lvj-overptn}.  
An overpartition is a regular partition in which the first occurrence of each part may be overlined.  
Then, a wave of results was the overpartition analogues of (regular) partition identities.  

Lovejoy found two cases in Rogers-Ramanujan-Gordon identities for overpartitions 
\cite{Lvj-Gordon-overptn}, and then the theorem for all cases was given by 
Chen, Sang and Shi \cite{Chen-etal-RRG-over}.  
Then, Corteel, Lovejoy and Mallet extended Bressoud's theorem to overpartitions 
in one case \cite{Ctl-etal-overptn-RR-evenmod}, and 
Chen, Sang and Shi showed the identities in all cases \cite{Chen-etal-overptn-Bressoud}.  

In this paper, we prove a theorem (Theorem \ref{thMain}) 
which contains all of the above theorems inside an infinite family of identities.  
In particular, we derive a $\pmod{d}$ analogue of Bressoud's $\pmod{2}$ condition.  
We show that the (regular) partition theorems naturally sit inside the overpartition theorems.  
That is, the identities are not just stated in a single theorem, but they are proved simultaneously.  

In all of the cited works, 
the approach is to work with multiplicity conditions on partitions on the one hand, 
and to work with variants of known series on the other hand, 
and eventually arguing that the functional equations and the initial conditions match.  
Two separate lines of computations are reconciled at the end.  
Here in contrast, we construct series from scratch using those functional equations 
that are implied by the recurrences that come from the definition of classes of partitions.  
In other words, the construction is linear.  
This is a modification of the method given in \cite{K-Andrews-Stayla}.  

The construction gives unusual companion identities as well.  
Unusual in the sense that while the known results assert equality between two partition counters, 
the companion identities give equalities between \emph{combinations} of partition counters.  

In Section \ref{secPrelim}, we collect all definitions for the sake of completeness, 
give a few examples and some straightforward facts involving infinite series and products.  
In Section \ref{secMain}, the main results are stated and proven.  
Finally, in Section \ref{secFuture}, the construction is explained 
in detail with its extent and shortcomings.  
We conclude with some directions for future research and open questions.  

\section{Preliminaries}
\label{secPrelim}

\begin{defn}
\label{defPtn}
  A partition of a non-negative integer $n$ 
  is a non-increasing sum of positive integers 
  \begin{equation*}
    n = \lambda_1 + \lambda_2 + \cdots + \lambda_m
  \end{equation*}
  where $\lambda_1 \geq \lambda_2 \geq \cdots \geq \lambda_m > 0$.  
  The number of parts $m$ is also known as the length of the partition.  
  
  Alternatively, one can write
  \begin{equation*}
    n = 1 f_1 + 2 f_2 + 3 f_3 + \cdots
  \end{equation*}
  for the same partition, where $f_i$ denotes 
  the number of occurrences, or the frequency, of $i$
  among $\lambda_1$, $\lambda_2$, \ldots, $\lambda_m$.  
\end{defn}
Obviously, only finitely many of the $f_i$ can be nonzero.  
For example the non-increasing sum
\begin{equation}
\label{exPtn}
  5+5+3+3+2+1+1+1
\end{equation}
is a partition of 21, 
where
\begin{equation*}
  f_1 = 3, \quad f_2 = 1, \quad f_3 = 2, \quad f_4 = 0, \quad f_5 = 2, \quad 
  \textrm{ and } f_i = 0 \textrm{ for } i \geq 6.  
\end{equation*}

\begin{defn}
\label{defOverPtn}
  An overpartition of a non-negative integer $n$
  is a partition of $n$ in which the first occurrence of each part may be overlined.  
  One can write
  \begin{equation*}
    n = 1 f_1 + 1 f_{\overline{1}} + 2 f_2 + 2 f_{\overline{2}} + 3 f_3 + 3 f_{\overline{3}} + \cdots
  \end{equation*}
  where $f_i$ denotes the number of occurrences, or the frequency, of $i$ (non-overlined), 
  and $f_{\overline{i}}$ denotes that of $\overline{i}$ (overlined).  
\end{defn}
Again, only finitely many of $f_i$ or $f_{\overline{i}}$s can be non-zero.  
In addition, $f_{\overline{i}}$'s may be 0 or 1 only.  
For example, 
\begin{equation}
\label{exOverPtn}
  8+8+7+7+5+5+\overline{4}+3+3+\overline{2}+\overline{1}+1
\end{equation}
is an overpartition of 54 where 
\begin{align*}
  & f_1 = 1, \;
  f_{\overline{1}} = 1, \;
  f_2 = 0, \;
  f_{\overline{2}} = 1, \;
  f_3 = 2, \;
  f_{\overline{3}} = 0, \;
  f_4 = 0, \;
  f_{\overline{4}} = 1, \;
  f_5 = 2, \; 
  f_{\overline{5}} = 0, \; \\
  & f_6 = 0, \;
  f_{\overline{6}} = 0, \;
  f_7 = 2, \;
  f_{\overline{7}} = 0, \;
  f_8 = 2, \;
  f_{\overline{8}} = 0, \;
  \textrm{ and } f_i = f_{\overline{i}} = 0 \textrm{ for } i \geq 9.  
\end{align*}

\begin{defn}
\label{defRhoStat}
  Given an overpartition and an arbitrary positive integer $i$ 
  that need not occur in the overpartition, 
  \begin{equation*}
    \rho(i) = \sum_{j = 1}^i (-1)^j f_{\overline{j}}.  
  \end{equation*}
  In other words, $\rho(i)$ is the signed sum of number of occurrences of 
  overlined parts that are less than or equal to $i$.  
\end{defn}
For instance, the overpartition \eqref{exOverPtn} has
\begin{equation*}
  \rho(1) = -1, \;
  \rho(2) = 0, \;
  \rho(3) = 0, \;
  \textrm{ and } \rho(i) = 1, \textrm{ for } i \geq 4.  
\end{equation*}

The unsigned sum $V(i) = \sum_{j = 1}^i f_{\overline{j}}$ was 
defined in \cite{Ctl-etal-overptn-RR-evenmod} 
and used in \cite{Chen-etal-overptn-Bressoud, Ctl-etal-overptn-RR-evenmod}.  
Again, as an example, the overpartition \eqref{exOverPtn} has 
\begin{equation*}
  V(1) = 1, \;
  V(2) = 2, \;
  V(3) = 2, \;
  \textrm{ and } V(i) = 3, \textrm{ for } i \geq 4.  
\end{equation*}

For ease of reference, we collected the partition counters 
used in this paper in the following definition.  

\begin{defn}
\label{defEnumerants}
  Let $n$, $m$, $k$, $a$, $d$, $s$ be non-negative integers such that 
  \[
    k \geq 2, \quad 
    1 \leq a \leq k, \quad 
    1 \leq d \leq k, \quad 
    0 \leq s \leq d-1.  
  \]
  
  \begin{tabular}{rcp{13cm}}
    ${}_dA_{k, a}(n)$ & = & 
      the number of partitions of $n$
      using parts $\not\equiv 0, \pm a \pmod{2k+2-d}$, 
      unless $2a = 2k + 2 - d$. 
  \end{tabular}
  
  \begin{tabular}{rcp{13cm}}
    ${}_d\overline{A}_{k, a}(n)$ & = & 
       number of overpartitions of $n$
       where non-overlined parts $\not\equiv 0, \pm a$ $\pmod{2k+1-d}$, 
	if $2a \neq 2k+1-d$, 
	
       number of overpartitions of $n$
       where all parts $\not\equiv 0 \pmod{k + (1-d)/2}$, 
	if $2a = 2k+1-d$.  
  \end{tabular}
  
  \begin{tabular}{rcp{13cm}}
    ${}_dB^s_{k, a}(n)$ & = & 
      number of partitions of $n$ such that 
      $\quad f_i + f_{i+1} < k$, 
      $\quad f_1 < a$, 
      
      and for $\delta = 1, 2, \ldots, d-1$, 
      if $f_i + f_{i+1} = k-\delta$, 
      
      then $a + s - 1 - f_{\textrm{odd}} \equiv 0, 1, \ldots, \delta-1 \pmod{d}$, 
      
      where $f_{\textrm{odd}} = f_i$ if $i$ is odd, 
      and $f_{\textrm{odd}} = f_{i+1}$ if $i$ is even.  
  \end{tabular}
  
  \begin{tabular}{rcp{13cm}}
    ${}_d\overline{B}^s_{k, a}(n)$ & = & 
      number of overpartitions of $n$ such that 
      $\quad f_i + f_{\overline{i}} + f_{i+1} < k$, 
      $\quad f_1 < a$, 
      
      and for $\delta = 1, 2, \ldots, d-1$, 
      if $f_i + f_{\overline{i}} + f_{i+1} = k-\delta$, 
      
      then $a + s - 1 - f_{\textrm{odd}}  - \rho(i) \equiv 0, 1, \ldots, \delta-1 \pmod{d}$, 
      
      where $f_{\textrm{odd}} = f_i + f_{\overline{i}}$ if $i$ is odd, 
      and $f_{\textrm{odd}} = f_{i+1}$ if $i$ is even.  
  \end{tabular}
  
  \begin{tabular}{rcp{13cm}}
    ${}_db^s_{k, a}(m, n)$ & = & 
      number of partitions of $n$ counted by ${}_dB^s_{k, a}(n)$ 
      which have exactly $m$ parts.  
  \end{tabular}
  
  \begin{tabular}{rcp{13cm}}
    ${}_d\overline{b}^s_{k, a}(m, n)$ & = & 
      number of overpartitions of $n$ counted by ${}_d\overline{B}^s_{k, a}(n)$ 
      which have exactly $m$ parts.  
  \end{tabular}
\end{defn}
It is straightforward to check that
the partition \eqref{exPtn} is counted by ${}_4b^1_{5, 3}(8, 21)$, 
and the overpartition \eqref{exOverPtn} is counted by ${}_4\overline{b}^1_{5, 3}(12, 54)$.  

We follow the $q$-series notation in \cite{GR}.  
Below, $n$ is a non-negative integer, $z$, $z_1$, $z_2$, \ldots $z_r$ are arbitrary indeterminates, 
and $q$ is a complex number such that $\vert q \vert < 1$.   
The condition on $q$ ensures the absolute convergence of all products and series in this note.  
\begin{align*}
    (z; q)_n & = (1 - z)(1 - zq) \cdots (1 - zq^{n-1}),  \\
    (z; q)_\infty & = \lim_{n \to \infty} (z; q)_n,  \\
    (z_1, z_2, \ldots, z_r; q)_\infty & = (z_1; q)_\infty (z_2; q)_\infty \cdots (z_r; q)_\infty.  
\end{align*}

\begin{prop}
\label{propInfProd}
  Let the parameters and enumerants be as in Definition \ref{defEnumerants}.  Then, 
  \begin{align*}
    \sum_{n \geq 0} {}_dA_{k, a}(n) q^n & = 
    \frac{ ( q^{a}, q^{(2k+2-d)-a}, q^{(2k+2-d)}; q^{(2k+2-d)})_\infty }
	{ ( q; q)_\infty }, \\
    \sum_{n \geq 0} {}_d\overline{A}_{k, a}(n) q^n & = 
    \frac{ (-q; q)_\infty ( q^{a}, q^{(2k+1-d)-a}, q^{(2k+1-d)}; q^{(2k+1-d)})_\infty }
	{ ( q; q)_\infty }.  
  \end{align*}
\end{prop}
\begin{proof}
  The only case entitled to justification is $2a = 2k+1-d$ in the second identity, 
  the others being evident.  
  In that case, let $\kappa = a = k + (1 - d)/2$ so that the infinite product 
  on the right hand side becomes
  \begin{align*}
    & \frac{ (-q; q)_\infty ( q^{\kappa}, q^{\kappa}, q^{2 \kappa}; q^{2 \kappa} )_\infty }
	{ (q; q)_\infty }
    = \frac{(-q; q)_\infty  ( q^{\kappa}; q^{\kappa} )_\infty ( q^{\kappa}; q^{2 \kappa} )_\infty }
	    { (q; q)_\infty } 
    = \frac{(-q; q)_\infty  ( q^{\kappa}; q^{\kappa} )_\infty }
	    { ( - q^{\kappa}; q^{\kappa} )_\infty (q; q)_\infty } \\
    = & \sum_{n \geq 0} {}_d\overline{A}_{k, k + (1 - d)/2}(n) \; q^n.  
  \end{align*}
  We used Euler's identity in the penultimate equation \cite[equation (1.2.5)]{Andrews-bluebook}.  
\end{proof}

\section{Main Results}
\label{secMain}

\begin{lemma}
\label{lemmaRecurrences}
  Let the parameters and enumerants be as in Definiton \ref{defEnumerants}.  Then, 
  \begin{align}
  \label{eqRecRegular}
    {}_db^s_{k, a}(m,n) 
    = & {}_db^{s+1}_{k, a-1}(m,n) + {}_db^0_{k, k-a+1-s}(m-a+1,n-m), \\
  \label{eqRecOver}
    {}_d\overline{b}^s_{k, a}(m,n) 
    = & {}_d\overline{b}^{s+1}_{k, a-1}(m,n) + {}_d\overline{b}^0_{k, k-a+1-s}(m-a+1,n-m) \\
  \nonumber
	& + {}_d\overline{b}^0_{k, k-a-s}(m-a,n-m),  \\
  \label{eqInitNoPart}
    {}_db^s_{k, a}(0,n) = & {}_d\overline{b}^s_{k, a}(0,n)
    = \begin{cases}
       1, & \textrm{ if } n = 0 \\
       0, & \textrm{ if } n > 0
      \end{cases}, \\
  \label{eqInitaEqZero}
    {}_db^s_{k, 0}(m,n) = & {}_d\overline{b}^s_{k, 0}(m,n) = 0.  
  \end{align}
  In the right-hand-sides of equations \eqref{eqRecRegular} and \eqref{eqRecOver}, 
  the superscript $s$ is understood as a residue class modulo $d$, 
  so when $s = d-1$, $s+1 = 0$.  
\end{lemma}
\begin{proof}
  Equation \eqref{eqInitNoPart} says that the only partition or overpartition with no parts
  is the empty partition or empty overpartition of zero.  
  Equation \eqref{eqInitaEqZero} is true because there are no partitions or overpartitions 
  where the number of occurrences of (non-overlined) 1 is strictly less than zero.  
  A part cannot occur a negative number of times.  
  
  We will prove equation \eqref{eqRecOver}, 
  and then argue that it implies equation \eqref{eqRecRegular}.  
  
  Let   
  \begin{align*}
   {}_d\overline{\mathcal{B}}^s_{k, a}(m, n) 
   = & \textrm{ the collection of overpartitions enumerated by } {}_d\overline{b}^s_{k, a}(m, n), \\
   \mathcal{U} 
   = & \textrm{overpartitions in } {}_d\overline{\mathcal{B}}^s_{k, a}(m, n) 
    \textrm{ for which } f_1 < a-1, \\
   \mathcal{V} 
   = & \textrm{overpartitions in } {}_d\overline{\mathcal{B}}^s_{k, a}(m, n) 
    \textrm{ for which } f_1 = a-1 \textrm{ and } f_{\overline{1}} = 0, \\
   \mathcal{W} 
   = & \textrm{overpartitions in } {}_d\overline{\mathcal{B}}^s_{k, a}(m, n) 
    \textrm{ for which } f_1 = a-1 \textrm{ and } f_{\overline{1}} = 1. \\
  \end{align*}
  First, ${}_d\overline{\mathcal{B}}^s_{k, a}(m, n)$ is the disjoint union 
  of $\mathcal{U}$, $\mathcal{V}$, and $\mathcal{W}$, 
  because the conditions $f_1 < a$ and $f_{\overline{1}} = $ 0 or 1  
  that are satisfied by all overpartitions in ${}_d\overline{\mathcal{B}}^s_{k, a}(m, n)$ 
  can be divided into the mutually exclusive and complementary cases 
  stipulated by $\mathcal{U}$, $\mathcal{V}$, and $\mathcal{W}$.  
  
  Next, an overpartition $\eta \in {}_d\overline{\mathcal{B}}^s_{k, a}(m, n)$ 
  has $f_1 < a-1$ if and only if $\eta \in {}_d\overline{\mathcal{B}}^{s+1}_{k, a-1}(m, n)$.  
  In the condition involving $f_i + f_{\overline{i}} + f_{i+1} = k - \delta$, 
  observe that the quantity $a + s \equiv (a - 1) + (s+1) \pmod{d}$ remains invariant.  
  So, $\mathcal{U}$ is in one-to-one correspondence with ${}_d\overline{\mathcal{B}}^{s+1}_{k, a-1}(m, n)$.  
  
  When an $\eta \in {}_d\overline{\mathcal{B}}^s_{k, a}(m, n)$ 
  has $f_1 = a-1$, and $f_{\overline{1}} = 0$, 
  we can produce $\widetilde{\eta}$ by deleting the 1's in $\eta$, 
  and subtracting 1 from all other parts.  
  Then $\widetilde{\eta}$ is an overpartition of $n-m$ with $m-a+1$ parts, 
  since deleting 1's also correspond to subtracting 1 from them.  
  
  In $\widetilde{\eta}$, call the frequencies of $i$ and $\overline{i}$ 
  $\widetilde{f_i}$ and $\widetilde{f_{\overline{i}}}$, respectively, 
  so $\widetilde{f}_i = f_{i+1}$ and $\widetilde{f}_{\overline{i}} = f_{\overline{i+1}}$.  
  $\widetilde{\eta}$ also satisfies $\widetilde{f}_i + \widetilde{f}_{\overline{i}} + \widetilde{f}_{i+1} < k$.  
  Given that if $f_1 + f_{\overline{1}} + f_2 = k - \delta$ then 
  $ a + s - 1 - f_1 - f_{\overline{1}} - \rho(1) \equiv 0, 1, \ldots, \delta - 1 \pmod{d}$; 
  since $f_1 = a-1$, $f_{\overline{1}} = 0$, $\rho(1) = 0$ in $\eta$, 
  $ s \equiv 0, 1, \ldots, \delta - 1 \pmod{d}$.  
  Because $s$ is fixed, $\delta$ can take values $s+1$, $s+2$, \ldots, $d-1$.  
  In other words, $f_1 + f_2 = a-1 + f_2$ is allowed to take the values 
  $k-s-1$, $k-s-2$, \ldots, $k-d-1$; 
  so $f_2 = \widetilde{f}_1 < k-a+1-s$.  
  This is why the latter lower index of the second summand on the right hand side of 
  equation \eqref{eqRecOver} is $k-a+1-s$.  
  
  Finally, $f_{i+1} + f_{\overline{i+1}} + f_{i+2} = k - \delta$ in $\eta$ whenever
  $\widetilde{f}_i + \widetilde{f}_{\overline{i}} + \widetilde{f}_{i+1} = k-\delta$ in $\widetilde{\eta}$.  
  Moreover, $f_{odd} + \widetilde{f}_{odd} = k-\delta$.  
  Because $f_{\overline{1}} = 0$ in $\eta$, $\widetilde{\rho}(i) = -\rho(i+1)$.  
  $\widetilde{\rho}$ gives the $\rho-$statistic in $\widetilde{\eta}$.  
  Subtracting 1 from parts switched parity, and thus switched the sign of the $\rho-$statistic.  
  We are now down to showing
  \begin{equation}
  \label{eqEta}
    a + s - 1 - f_{odd} - \rho(i+1) \equiv 0, 1, \ldots, \delta - 1 \pmod{d}
  \end{equation}
  if and only if 
  \begin{equation}
  \label{eqEtaTilde}
    (k-a+1-s) + 0 - 1 - \widetilde{f}_{odd} - \widetilde{\rho}(i) \equiv 0, 1, \ldots, \delta - 1 \pmod{d}
  \end{equation}
  for fixed $\delta$.  Making the substitutions per the above paragraph, 
  congruence \eqref{eqEtaTilde} is equivalent to 
  \begin{equation*}
    - a - s  + \delta + f_{odd} + \rho(i+1) \equiv 0, 1, \ldots, \delta - 1 \pmod{d}, 
  \end{equation*}
  which in turn is equivalent to congruence \eqref{eqEta} after multiplying by $-1$ 
  and adding $\delta-1$ to both sides.  
  Consequently, $\widetilde{\eta} \in {}_d\overline{\mathcal{B}}^0_{k, k-a+s-1}(m-a+1, m-n)$.  
  
  Conversely, each such $\widetilde{\eta} \in {}_d\overline{\mathcal{B}}^0_{k, k-a+s-1}(m-a+1, m-n)$ 
  yields an $\eta \in \mathcal{V}$ by adding 1 to each part and appending $(a-1)$ extra 1's.  
  This shows that there is a one-to-one correspondence 
  between ${}_d\overline{\mathcal{B}}^0_{k, k-a+s-1}(m-a+1, m-n)$ and $\mathcal{V}$.  
  
  This procedure shows a one-to-one correspondence between 
  ${}_d\overline{\mathcal{B}}^0_{k, k-a+s}(m-a, m-n)$ and $\mathcal{W}$ as well.  
  In this case, $\eta$ is losing $a$ parts, $(a-1)$ 1's and a $\overline{1}$;   
  and $\rho(i+1) = 1 - \widetilde{\rho}(i)$.  
  This establishes equation \eqref{eqRecOver}.  
  
  To see equation \eqref{eqRecRegular}, simply observe that 
  a (regular) partition is an overpartition with no overlined parts, 
  hence $\mathcal{W}$ is empty, and $\rho(i) = 0$ for all $i$ in $\eta$.  
\end{proof}

Making the reasonable assumption that ${}_db^{s}_{k, a}(m, n)$ and similar partition counters 
are zero when $n$ or $m$ is negative, 
one sees that the equations \eqref{eqRecRegular} - \eqref{eqInitaEqZero}
uniquely determine ${}_db^s_{k, a}(m,n)$'s and ${}_d\overline{b}^s_{k, a}(m,n)$'s.  
We then have the following corollary.  

\begin{cor}
\label{corFuncEq}
  Let the parameters and enumerants be as in Definition \ref{defEnumerants}.  
  Set 
  \begin{align*}
    {}_dF^{s}_{k, a}(x; q) = & \sum_{m, n \geq 0}
      {}_db^{s}_{k, a}(m, n) x^m q^n, \\
    {}_d\overline{F}^{s}_{k, a}(x; q) = & \sum_{m, n \geq 0}
      {}_d\overline{b}^{s}_{k, a}(m, n) x^m q^n.  
  \end{align*}
  Then, 
  \begin{align*}
    {}_dF^{s}_{k, a}(x; q)  - {}_dF^{s+1}_{k, a-1}(x; q) 
    = & (xq)^{a-1} {}_dF^{0}_{k, k-a+1-s}(xq; q), \\
    {}_d\overline{F}^{s}_{k, a}(x; q)  - {}_d\overline{F}^{s+1}_{k, a-1}(x; q) 
    = & (xq)^{a-1} {}_d\overline{F}^{0}_{k, k-a+1-s}(xq; q) 
      + (xq)^{a} {}_d\overline{F}^{0}_{k, k-a-s}(xq; q), \\
    {}_dF^{s}_{k, a}(0; q) 
    = & {}_d\overline{F}^{s}_{k, a}(0; q) = 1, \\
    {}_dF^{s}_{k, 0}(x; q) 
    = & {}_d\overline{F}^{s}_{k, 0}(x; q) = 0. 
  \end{align*}
  The listed functional equations and initial conditions uniquely determine 
  ${}_dF^{s}_{k, a}(x; q)$'s and ${}_d\overline{F}^{s}_{k, a}(x; q)$'s. 
\end{cor}

\begin{lemma}
\label{lemmaSeries}
  Let the parameters be as in Definition \ref{defEnumerants}.  Set
  \begin{align*}
    {}_dG^{s}_{k, a}(x; q) = \sum_{n \geq 0} 
	  {}_d\alpha[s]_n(x; q) q^{-na} + {}_d\beta[s]_n(x; q) (xq^{n+1})^{a},  \\
    {}_d\overline{G}^{s}_{k, a}(x; q) = \sum_{n \geq 0} 
	  {}_d\overline{\alpha}[s]_n(x; q) q^{-na} + {}_d\overline{\beta}[s]_n(x; q) (xq^{n+1})^{a}, 
  \end{align*}
  where
  \begin{align*}
    {}_d\alpha[s]_n(x; q) 
    = & \; \frac{ ((xq)^d; q^d)_\infty }{ (xq; q)_\infty } \; (-1)^n \; 
      \frac{ x^{(k+1-d)n} \; q^{(2k+2-d)n(n+1)/2} }{ (q^d; q^d)_n \; ( (xq^{n+1})^d; q^d )_\infty } \\
    & \times (q^{-n})^s \left( q^{dn} \; (xq)^{d-s} \frac{(1 - (xq)^s)}{(1 - (xq)^d)} 
      + \frac{(1 - (xq)^{d-s})}{(1 - (xq)^d)} \right),  \\
    {}_d\beta[s]_n(x; q) 
    = & \; - \; \frac{ ((xq)^d; q^d)_\infty }{ (xq; q)_\infty } \; (-1)^n \; 
      \frac{ x^{(k+1-d)n} \; q^{(2k+2-d)n(n+1)/2} }{ (q^d; q^d)_n \; ( (xq^{n+1})^d; q^d )_\infty } \\
    & \times (q^{-n})^{d-s} \left( \frac{(1 - (xq)^s)}{(1 - (xq)^d)} 
      + q^{dn} \; (xq)^{s} \; \frac{(1 - (xq)^{d-s})}{(1 - (xq)^d)} \right), 
  \end{align*}
  \begin{align*}
    {}_d\overline{\alpha}[s]_n(x; q) 
    = & \; \frac{ ((xq)^d; q^d)_\infty }{ (xq; q)_\infty } \; (-1)^n \; 
      \frac{ x^{(k+1-d)n} \; q^{(2k+1-d)n(n+1)/2} }{ (q^d; q^d)_n \; ( (xq^{n+1})^d; q^d )_\infty } 
      \; (-q; q)_n \; (-xq^{n+1}; q)_\infty \\
    & \times (q^{-n})^s \left( q^{dn} \; (xq)^{d-s} \frac{(1 - (xq)^s)}{(1 - (xq)^d)} 
      + \frac{(1 - (xq)^{d-s})}{(1 - (xq)^d)} \right), \\
    {}_d\overline{\beta}[s]_n(x; q) 
    = & \; - \; \frac{ ((xq)^d; q^d)_\infty }{ (xq; q)_\infty } \; (-1)^n \; 
      \frac{ x^{(k+1-d)n} \; q^{(2k+1-d)n(n+1)/2} }{ (q^d; q^d)_n \; ( (xq^{n+1})^d; q^d )_\infty } 
      \; (-q; q)_n \; (-xq^{n+1}; q)_\infty \\
    & \times (q^{-n})^{d-s} \left( \frac{(1 - (xq)^s)}{(1 - (xq)^d)} 
      + q^{dn} \; (xq)^{s} \; \frac{(1 - (xq)^{d-s})}{(1 - (xq)^d)} \right).  
  \end{align*}
  Then, 
  \begin{align}
  \label{eqSeriesFuncEqReg}
    {}_dG^{s}_{k, a}(x; q)  - {}_dG^{s+1}_{k, a-1}(x; q) 
    = & (xq)^{a-1} {}_dG^{0}_{k, k-a+1-s}(xq; q), \\
  \nonumber
    {}_d\overline{G}^{s}_{k, a}(x; q)  - {}_d\overline{G}^{s+1}_{k, a-1}(x; q) 
    = & (xq)^{a-1} {}_d\overline{G}^{0}_{k, k-a+1-s}(xq; q) 
      + (xq)^{a} {}_d\overline{G}^{0}_{k, k-a-s}(xq; q), \\
  \nonumber
    {}_dG^{s}_{k, a}(0; q) 
    = & {}_d\overline{G}^{s}_{k, a}(0; q) = 1, \\
  \label{eqSeriesInitaEqZero}
    {}_dG^{s}_{k, 0}(x; q) 
    = & {}_d\overline{G}^{s}_{k, 0}(x; q) = 0 
    \quad \textrm{ ( if } 2s \equiv 0 \pmod{d} \textrm{ )}.  
  \end{align}
  Here again, $s$ in the superscripts is understood as a residue class modulo $d$, 
  so when $s = d-1$, $s+1 = 0$ in the superscripts.  
\end{lemma}
\begin{proof}
  One verifies the following functional equations: 
  \begin{align*}
    {}_d\alpha[s]_n(x; q) \; (q^{-n})^{a} - {}_d\alpha[s+1]_n(x; q) \; (q^{-n})^{a-1}
    = & \; (xq)^{a-1} \; {}_d\beta[0]_{n-1}(xq; q) \; (xq^{n+1})^{k-a+1-s}, \\
    {}_d\overline{\alpha}[s]_n(x; q) \; (q^{-n})^{a} - {}_d\overline{\alpha}[s+1]_n(x; q) \; (q^{-n})^{a-1} 
    = & \; (xq)^{a-1} \; {}_d\overline{\beta}[0]_{n-1}(xq; q) \; (xq^{n+1})^{k-a+1-s} \\
      & + \; (xq)^{a} \; {}_d\overline{\beta}[0]_{n-1}(xq; q) \; (xq^{n+1})^{k-a-s}, \\
    {}_d\beta[s]_n(x; q) \; (xq^{n+1})^{a} - {}_d\beta[s+1]_n(x; q) \; (xq^{n+1})^{a-1}
    = & \; (xq)^{a-1} \; {}_d\alpha[0]_{n}(xq; q) \; (q^{-n})^{k-a+1-s}, \\
    {}_d\overline{\beta}[s]_n(x; q) \; (xq^{n+1})^{a} - {}_d\overline{\beta}[s+1]_n(x; q) \; (xq^{n+1})^{a-1} 
    = & \; (xq)^{a-1} \; {}_d\overline{\alpha}[0]_{n}(xq; q) \; (q^{-n})^{k-a+1-s} \\
      & + \; (xq)^{a} \; {}_d\overline{\alpha}[0]_{n}(xq; q) \; (q^{-n})^{k-a-s}
  \end{align*}
  for $s = 0, \ldots, d-2$, and
  \begin{align*}
    {}_d\alpha[d-1]_n(x; q) \; (q^{-n})^{a} - {}_d\alpha[0]_n(x; q) \; (q^{-n})^{a-1}
    = & \; (xq)^{a-1} \; {}_d\beta[0]_{n-1}(xq; q) \; (xq^{n+1})^{k-a+2-d}, \\
    {}_d\overline{\alpha}[d-1]_n(x; q) \; (q^{-n})^{a} - {}_d\overline{\alpha}[0]_n(x; q) \; (q^{-n})^{a-1} 
    = & \; (xq)^{a-1} \; {}_d\overline{\beta}[0]_{n-1}(xq; q) \; (xq^{n+1})^{k-a+2-d} \\
      & + \; (xq)^{a} \; {}_d\overline{\beta}[0]_{n-1}(xq; q) \; (xq^{n+1})^{k-a+1-d}, \\ 
    {}_d\beta[d-1]_n(x; q) \; (xq^{n+1})^{a} - {}_d\beta[0]_n(x; q) \; (xq^{n+1})^{a-1}
    = & \; (xq)^{a-1} \; {}_d\alpha[0]_{n}(xq; q) \; (q^{-n})^{k-a+2-d}, \\
    {}_d\overline{\beta}[d-1]_n(x; q) \; (xq^{n+1})^{a} - {}_d\overline{\beta}[0]_n(x; q) \; (xq^{n+1})^{a-1} 
    = & \; (xq)^{a-1} \; {}_d\overline{\alpha}[0]_{n}(xq; q) \; (q^{-n})^{k-a+2-d} \\
      & + \; (xq)^{a} \; {}_d\overline{\alpha}[0]_{n}(xq; q) \; (xq^{-n})^{k-a+1-d}.  
  \end{align*}
  These are straightforward calculations.  
  The two groups of recurrences are separated for emphasis on how $s$ is interpreted.  
  Then one verifies the following:  
  \begin{align*}
    {}_d\alpha[s]_0(0; q) = & {}_d\overline{\alpha}[s]_0(0; q) = 1, \\
    {}_d\alpha[s]_n(x; q) = & - {}_d\beta[s]_n(x; q), \\
    {}_d\overline{\alpha}[s]_n(x; q) = & - {}_d\overline{\beta}[s]_n(x; q).  
  \end{align*}
  The last two identities hold only when $s = 0$ or $2s = d$, hence
  the condition $2s \equiv 0 \pmod {d}$.  
\end{proof}

\begin{cor}
\label{corSeriesGenFunc}
  Let ${}_dF^{s}_{k, a}(x; q)$, ${}_d\overline{F}^{s}_{k, a}(x; q)$, 
  ${}_dG^{s}_{k, a}(x; q)$, and ${}_d\overline{G}^{s}_{k, a}(x; q)$ 
  be given as in Corollary \ref{corFuncEq} and Lemma \ref{lemmaSeries}.  
  Then, 
  \begin{equation*}
    {}_dF^{s}_{k, a}(x; q) = {}_dG^{s}_{k, a}(x; q)
  \end{equation*}
  when $2(a+s) \equiv 2 (k+1) \equiv 0 \pmod{d}$, and 
  \begin{equation*}
    {}_d\overline{F}^{s}_{k, a}(x; q) = {}_d\overline{G}^{s}_{k, a}(x; q)  
  \end{equation*}
  when $d = 1$ or $d = 2$.  
\end{cor}
\begin{proof}
  Since ${}_dF^{s}_{k, a}(x; q)$ and ${}_dG^{s}_{k, a}(x; q)$ satisfy 
  the same functional equations and same initial conditions 
  when
  \begin{equation*}
   2(a+s) \equiv 0 \pmod{d} \textrm{, and } 2(k-a+1-s + 0) \equiv 0 \pmod{d}
  \end{equation*}
  the first identity follows.  
  The latter congruence is necessary because 
  the series on the right hand side of equation \eqref{eqSeriesFuncEqReg} 
  is subject to the same initial conditions.  
  
  Similar considerations for ${}_d\overline{G}^{s}_{k, a}(x; q)$
  bring 
  \begin{equation*}
   2(a+s) \equiv 0 \pmod{d}, \; 
   2(k-a+1-s + 0) \equiv 0 \pmod{d} \textrm{, and }
   2(k-a-s + 0) \equiv 0 \pmod{d}.  
  \end{equation*}
  The last two congruences imply $2 \equiv 0 \pmod{d}$, 
  forcing $d = 1$ or $d = 2$.  
\end{proof}

\begin{prop}
\label{propJTP}
  Let the parameters be as in Definition \ref{defEnumerants}, 
  and ${}_dG^{s}_{k, a}(x; q)$ and ${}_d\overline{G}^{s}_{k, a}(x; q)$ 
  be defined as in Lemma \ref{lemmaSeries}.  Then,
  \begin{align*}
    {}_dG^{s}_{k, a}(1; q) = & \; 
      \frac{ (q^{d-s} - q^{d}) }{ (1 - q^d) } \; 
      \frac{ ( q^{a+s-d}, q^{(2k+2-d)-(a+s-d)}, q^{(2k+2-d)}; q^{(2k+2-d)})_\infty }{ (q; q)_\infty } \\
    & + \frac{ (1 - q^{d-s}) }{ (1 - q^d) } \; 
      \frac{ ( q^{a+s}, q^{(2k+2-d)-(a+s)}, q^{(2k+2-d)}; q^{(2k+2-d)})_\infty }{ (q; q)_\infty } \\
    = & \; 
      \frac{ (q^{a+s} - q^{a}) }{ (1 - q^d) } \; 
      \frac{ ( q^{d-a-s}, q^{(2k+2-d)-(d-a-s)}, q^{(2k+2-d)}; q^{(2k+2-d)})_\infty }{ (q; q)_\infty } \\
    & + \frac{ (1 - q^{d-s}) }{ (1 - q^d) } \; 
      \frac{ ( q^{a+s}, q^{(2k+2-d)-(a+s)}, q^{(2k+2-d)}; q^{(2k+2-d)})_\infty }{ (q; q)_\infty }, \\
    {}_d\overline{G}^{s}_{k, a}(1; q) = & \; 
      \frac{ (q^{d-s} - q^{d}) }{ (1 - q^d) } \; 
      \frac{ ( -q; q)_\infty \; ( q^{a+s-d}, q^{(2k+1-d)-(a+s-d)}, q^{(2k+1-d)}; q^{(2k+1-d)})_\infty }
	  { (q; q)_\infty } \\
    & + \frac{ (1 - q^{d-s}) }{ (1 - q^d) } \; 
      \frac{ ( -q; q)_\infty \; ( q^{a+s}, q^{(2k+1-d)-(a+s)}, q^{(2k+1-d)}; q^{(2k+1-d)})_\infty }
	  { (q; q)_\infty }.  
  \end{align*}
\end{prop}
\begin{proof}
  We will demonstrate the former string of identities only, 
  the latter is completely analogous.  
  \begin{align*}
    {}_dG^{s}_{k, a}(1; q) = & \sum_{n \geq 0} 
	  {}_d\alpha[s]_n(1; q) q^{-na} + {}_d\beta[s]_n(1; q) (q^{n+1})^{a}
  \end{align*}
  \begin{align*}
    = \frac{ 1 }{ (1 - q^d) \; ( q; q)_\infty } \; 
    & \sum_{n \geq 0} (-1)^n q^{(2k+2-d)n(n+1)/2} \; q^{-an} 
      \; \left[ q^{-sn} \left( q^{dn+d-s} (1 - q^s) + (1 - q^{d-s}) \right) \right] \\
    - & (-1)^n q^{(2k+2-d)n(n+1)/2} \; q^{a(n+1)} 
      \; \left[ q^{-n(d-s)} \left( (1 - q^s) + q^{dn+s} (1 - q^{d-s}) \right) \right]
  \end{align*}
  \begin{align*}
    = \frac{ 1 }{ (1 - q^d) \; ( q; q)_\infty } \; 
    & \sum_{n \geq 0} (-1)^n q^{(2k+2-d)n(n+1)/2} \; q^{-an} 
      \; \left[ q^{-n(s-d)} (q^{d-s} - q^d) + q^{-ns} (1 - q^{d-s}) \right] \\
    - & (-1)^n q^{(2k+2-d)n(n+1)/2} \; q^{a(n+1)} 
      \; \left[ q^{(n+1)(s-d)} (q^{d-s} - q^d) + q^{(n+1)s} (1 - q^{d-s}) \right]
  \end{align*}
  \begin{align*}
    = \frac{ 1 }{ (1 - q^d) \; ( q; q)_\infty } \; 
     \bigg\{ & (q^{d-s} - q^d) \sum_{n \geq 0} (-1)^n q^{(2k+2-d)n(n+1)/2} \; q^{-n(a+s-d)} \\
     + & (q^{d-s} - q^d) \sum_{n \geq 0} (-1)^n q^{(2k+2-d)n(n+1)/2} \; q^{(n+1)(a+s-d)} \\
     + & (1 - q^{d-s}) \sum_{n \geq 0} (-1)^n q^{(2k+2-d)n(n+1)/2} \; q^{-n(a+s)} \\
     + & (1 - q^{d-s}) \sum_{n \geq 0} (-1)^n q^{(2k+2-d)n(n+1)/2} \;  q^{(n+1)(a+s)} \bigg\}
  \end{align*}
  Now substitute $n \leftarrow -n$ in the first and third sums, 
  $n \leftarrow n-1$ in the second and the fourth.  
  \begin{align*}
    = \frac{ 1 }{ (1 - q^d) \; ( q; q)_\infty } \; 
     \bigg\{ & (q^{d-s} - q^d) \sum_{n \leq 0} (-1)^n q^{(2k+2-d)n(n-1)/2} \; q^{n(a+s-d)} \\
     + & (q^{d-s} - q^d) \sum_{n \geq 1} (-1)^n q^{(2k+2-d)n(n-1)/2} \; q^{n(a+s-d)} \\
     + & (1 - q^{d-s}) \sum_{n \leq 0} (-1)^n q^{(2k+2-d)n(n-1)/2} \; q^{n(a+s)} \\
     + & (1 - q^{d-s}) \sum_{n \geq 1} (-1)^n q^{(2k+2-d)n(n-1)/2} \;  q^{n(a+s)} \bigg\}
  \end{align*}
  \begin{align*}
    = \frac{ 1 }{ (1 - q^d) \; ( q; q)_\infty } \; 
     \bigg\{ & (q^{d-s} - q^d) \sum_{n = -\infty}^\infty (-1)^n q^{(2k+2-d)n(n-1)/2} \; q^{n(a+s-d)} \\
     + & (1 - q^{d-s}) \sum_{n = -\infty}^\infty (-1)^n q^{(2k+2-d)n(n-1)/2} \;  q^{n(a+s)} \bigg\}
  \end{align*}
  Finally use Jacobi's triple product identity \cite[equation (1.6.1)]{GR} 
  to obtain the first identity in the former string of equations.  
  
  For the second identity there, 
  one just needs to notice that 
  \begin{equation*}
    (q^{d-s} - q^d) (1 - q^{a+s-d}) = (q^{a+s} - q^a)(1 - q^{d-(s+a)})
  \end{equation*}
  implies
  \begin{align*}
    & \frac{(q^{d-s} - q^d)}{(1 - q^d)} \; 
    ( q^{(a+s-d)}, q^{(2k+2-d)-(a+s-d )} ; q^{(2k+2-d)})_\infty \\
    & = \frac{(q^{s+a} - q^a)}{(1 - q^d)} \; 
    ( q^{(d-a-s)}, q^{(2k+2-d)-(d-a-s)} ; q^{(2k+2-d)})_\infty.  
  \end{align*}
\end{proof}

We had better recall part of Definition \ref{defEnumerants} in the main result.  
\begin{theorem}
\label{thMain}
  Let $n$, $m$, $k$, $a$, $d$, $s$ be non-negative integers such that 
  \[
    k \geq 2, \quad 
    1 \leq a \leq k, \quad 
    1 \leq d \leq k, \quad 
    0 \leq s \leq d-1.  
  \]
  
  \begin{tabular}{rcp{13cm}}
    ${}_dA_{k, a}(n)$ & = & 
      the number of partitions of $n$
      using parts $\not\equiv 0, \pm a \pmod{2k+2-d}$, 
      unless $2a = 2k + 2 - d$. 
  \end{tabular}
  
  \begin{tabular}{rcp{13cm}}
    ${}_d\overline{A}_{k, a}(n)$ & = & 
       number of overpartitions of $n$
       where non-overlined parts $\not\equiv 0, \pm a$ $\pmod{2k+1-d}$, 
	if $2a \neq 2k+1-d$, 
	
       number of overpartitions of $n$
       where all parts $\not\equiv 0 \pmod{k + (1-d)/2}$, 
	if $2a = 2k+1-d$.  
  \end{tabular}
  
  \begin{tabular}{rcp{13cm}}
    ${}_dB^s_{k, a}(n)$ & = & 
      number of partitions of $n$ such that 
      $\quad f_i + f_{i+1} < k$, 
      $\quad f_1 < a$, 
      
      and for $\delta = 1, 2, \ldots, d-1$, 
      if $f_i + f_{i+1} = k-\delta$, 
      
      then $a + s - 1 - f_{\textrm{odd}} \equiv 0, 1, \ldots, \delta-1 \pmod{d}$, 
      
      where $f_{\textrm{odd}} = f_i$ if $i$ is odd, 
      and $f_{\textrm{odd}} = f_{i+1}$ if $i$ is even.  
  \end{tabular}
  
  \begin{tabular}{rcp{13cm}}
    ${}_d\overline{B}^s_{k, a}(n)$ & = & 
      number of overpartitions of $n$ such that 
      $\quad f_i + f_{\overline{i}} + f_{i+1} < k$, 
      $\quad f_1 < a$, 
      
      and for $\delta = 1, 2, \ldots, d-1$, 
      if $f_i + f_{\overline{i}} + f_{i+1} = k-\delta$, 
      
      then $a + s - 1 - f_{\textrm{odd}}  - \rho(i) \equiv 0, 1, \ldots, \delta-1 \pmod{d}$, 
      
      where $f_{\textrm{odd}} = f_i + f_{\overline{i}}$ if $i$ is odd, 
      and $f_{\textrm{odd}} = f_{i+1}$ if $i$ is even.  
  \end{tabular}
  
  Then, in case $s = 0$, 
  \begin{equation}
  \label{eqMainResultNiceReg}
    {}_dA_{k, a}(n) = {}_dB^0_{k, a}(n)
  \end{equation}
  when $2a \equiv 2(k+1) \equiv 0 \pmod{d}$, 
  
  \begin{equation*}
    {}_d\overline{A}_{k, a}(n) = {}_d\overline{B}^0_{k, a}(n)
  \end{equation*}
  when $d = 1$ or $d= 2$, 
  
  and in case $s \neq 0$, 
  \begin{align}
  \label{eqMainResultUglyRegI}
    & {}_dA_{k, a+s}(n) - {}_dA_{k, a+s}(n - (d-s)) 
    + {}_dA_{k, a+s-d}(n-(d-s)) - {}_dA_{k, a+s-d}(n-d) \\
    \nonumber
    & = {}_dB^s_{k, a}(n) - {}_dB^s_{k, a}(n-d)
  \end{align}
  when $2(a+s) \equiv 2(k+1) \equiv 0 \pmod{d}$, $2(a+s) \neq 2k + 2 + d$, 
  and $d < a + s$, 
  
  \begin{align*}
    {}_dA_{k, a+s}(n) - {}_dA_{k, a+s}(n - (d-s))
    = {}_dB^s_{k, a}(n) - {}_dB^s_{k, a}(n-d)
  \end{align*}
  when $2(a+s) \equiv 2(k+1) \equiv 0 \pmod{d}$, $2(a+s) \neq 2k + 2 + d$, 
  and $d = a + s$, 
  
  \begin{align*}
    & {}_dA_{k, a+s}(n) - {}_dA_{k, a+s}(n - (d-s)) 
    + {}_dA_{k, d-a-s}(n-(a+s)) - {}_dA_{k, d-a-s}(n-a) \\
    & = {}_dB^s_{k, a}(n) - {}_dB^s_{k, a}(n-d)
  \end{align*}
  when $2(a+s) \equiv 2(k+1) \equiv 0 \pmod{d}$, $2(a+s) \neq 2k + 2 + d$, 
  and $d > a + s$, 

  \begin{align*}
    & {}_d\overline{A}_{k, a+s}(n) - {}_d\overline{A}_{k, a+s}(n - (d-s)) 
    + {}_d\overline{A}_{k, a+s-d}(n-(d-s)) - {}_d\overline{A}_{k, a+s-d}(n-d) \\
    & = {}_d\overline{B}^s_{k, a}(n) - {}_d\overline{B}^s_{k, a}(n-d)
  \end{align*}
  when $d = 1$ or $d = 2$, 
  and $d < a + s$, 

  \begin{align*}
    {}_d\overline{A}_{k, a+s}(n) - {}_d\overline{A}_{k, a+s}(n - (d-s)) 
    = {}_d\overline{B}^s_{k, a}(n) - {}_d\overline{B}^s_{k, a}(n-d)
  \end{align*}
  when $d = 1$ or $d = 2$, 
  and $d = a + s$.  
\end{theorem}

The case $d > a+s$ is vacuous for $d =$ 1 or 2, 
and $s \neq 0$.  

\begin{proof}
  Observe that equations \eqref{eqMainResultNiceReg} and \eqref{eqMainResultUglyRegI}
  are equivalent to the identity of the generating functions
  \begin{equation*}
    \frac{(q^{d-s} - q^{d})}{(1 - q^d)} \sum_{n \geq 0} {}_dA_{k, a+s-d}(n)
    + \frac{(1 - q^{d-s})}{(1 - q^d)} \sum_{n \geq 0} {}_dA_{k, a+s}(n)
    = \sum_{n \geq 0} {}_dB^s_{k, a}(n) 
  \end{equation*}
  which is a consequence of Proposition \ref{propInfProd}, 
  Corollary \ref{corSeriesGenFunc}, and Proposition \ref{propJTP} combined.  
  The same line of reasoning applies to other equations \emph{mutatis mutandis}.  
\end{proof}


\section{Discussion and Further Research}
\label{secFuture}

When $d = 1$, the condition $\pmod{1}$ is vacuous, 
so Theorem \ref{thMain} yields Rogers-Ramanujan-Gordon identities \cite{RRG} (Theorem \ref{thRRG}), 
and Gordon's theorem for overpartitions \cite{Chen-etal-RRG-over, Lvj-Gordon-overptn}.  

For $d = 2$ and $s = 0$, the condition $2a \equiv 2(k+1) \equiv 0 \pmod{2}$ 
is satisfied by any $k$ or $a$. 
Also, the only possible value of $\delta$ is 1.  
So we just need to show the equivalence of 
\begin{equation*}
  i f_i + i f_{\overline{i}} + (i+1) f_{i+1} \equiv V(i) + a - 1 \pmod{2}
\end{equation*}
and
\begin{equation*}
  a - 1 - f_{\textrm{odd}} - \rho(i) \equiv 0 \pmod{2}.  
\end{equation*}
when $i f_i + i f_{\overline{i}} + (i+1) f_{i+1} = k-1$.  
This is a consequence of $1 \equiv -1 \pmod{2}$, 
and $i f_i + i f_{\overline{i}} + (i+1) f_{i+1} \equiv f_{\textrm{odd}} \pmod{2}$.  
We can switch to (regular) partitions from overpartitions 
as described in the proof of Lemma \ref{lemmaRecurrences}.  
So we recover Bressoud's all-moduli generalization of 
Rogers-Ramanujan-Gordon identities \cite{Br-RRG-AllMod} (Theorem \ref{thBressoud}), 
and that for overpartitions \cite{Chen-etal-overptn-Bressoud, Ctl-etal-overptn-RR-evenmod}
via Theorem \ref{thMain}.  

For all other combinations of $d$ and $s$, the results are new.  
The reader may find obtaining the $d = 2$, $s = 1$ cases 
from the $d = 2$, $s = 0$ cases of Theorem \ref{thMain} an amusing combinatorial exercise.  

The construction of the series in Lemma \ref{lemmaSeries} 
is a little different from the construction in \cite{K-Andrews-Stayla}.  
There, the construction progresses linearly 
from the definition to the series.  
Here in contrast, one begins with Bressoud's extension of 
Rogers-Ramanujan-Gordon identities \cite{Br-RRG-AllMod} 
or its counterpart for overpartitions \cite{Chen-etal-overptn-Bressoud, Ctl-etal-overptn-RR-evenmod},
and reasons that there must be a companion identity.  

For demonstration purposes, let's work with (regular) partitions.  
In the extension of Rogers-Ramanujan-Gordon due to Bressoud \cite{Br-RRG-AllMod}, 
the extra condition
\begin{equation*}
  f_{i} + f_{i+1} = k-1 
  \quad \Rightarrow \quad 
  i \, f_i + (i+1) f_{i+1} \equiv a - 1 \pmod{2}
\end{equation*}
on top of Gordon's condition has the obvious companion 
\begin{equation*}
  f_{i} + f_{i+1} = k-1 
  \quad \Rightarrow \quad 
  i \, f_i + (i+1) f_{i+1} \not\equiv a - 1 \pmod{2}.  
\end{equation*}
Then, one labels generating functions for 
partitions satisfying the former condition 
and partitions satisfying the latter condition.  
These are ${}_2F^0_{k, a}(x; q)$ and ${}_2F^1_{k, a}(x; q)$, respectively, 
in our notation (cf. the equivalence of conditions at the beginning of this section).  
After that, one constructs series ${}_2G^0_{k, a}(x; q)$ and ${}_2G^1_{k, a}(x; q)$ 
as instructed in \cite[sec. 4]{K-Andrews-Stayla}.  
One of the matrix equations one gets is 
\begin{equation*}
  \begin{bmatrix} q^{-n} & -1 \\ -1 & q^{-n} \end{bmatrix} \; 
  \begin{bmatrix} {}_2\alpha[0]_n(x; q) \\ {}_2\alpha[1]_n(x; q) \end{bmatrix}
  = (xq^{n+1})^k \; {}_2\beta[0]_{n-1}(xq; q) \;
  \begin{bmatrix} 1 \\ (xq^{n+1})^{-1} \end{bmatrix}
\end{equation*}
Those matrix equations have nice solutions that yield 
compatible series with the Jacobi's triple product identity \cite[equation (1.6.1)]{GR}, 
and one has the $d = 2$ case for (regular) partitions in Theorem \ref{thMain}.  

At this point, one has two choices.  
The first is to try and find general conditions $\pmod{d}$ which reduce to 
the conditions above when $d = 2$.  
And then, use (not automated so far) machinery of \cite{K-Andrews-Stayla} 
to construct series for partition generating functions.  
One always gets some series, but they may or may not be product identity friendly.  
This path has not been very fruitful in this setting.  

The other choice is, try and find product-identity friendly series first, 
which coincide with the known series in $d = 1$ and $d = 2$ cases 
(together with their functional equations), 
and then set those series as a departure point to produce the definitons.  
The above matrix equation may be extended to larger sizes as
\begin{equation*}
  \begin{bmatrix} q^{-n} & -1 & & \\ & q^{-n} & -1 & \\ 
      & & \ddots & \\ -1 & & & q^{-n} \end{bmatrix}
  \begin{bmatrix} {}_d\alpha[0]_n(x; q) \\ {}_d\alpha[1]_n(x; q) \\ 
      \vdots \\ {}_d\alpha[d-1]_n(x; q) \end{bmatrix}
  = (xq^{n+1})^k \; {}_d\beta[0]_{n-1}(xq; q) \;
  \begin{bmatrix} 1 \\ (xq^{n+1})^{-1} \\ \vdots \\ (xq^{n+1})^{-(d-1)} \end{bmatrix}
\end{equation*}
and 
\begin{equation*}
  \begin{bmatrix} xq^{n+1} & -1 & & \\ & xq^{n+1} & -1 & \\ 
      & & \ddots & \\ -1 & & & xq^{n+1} \end{bmatrix}
  \begin{bmatrix} {}_d\beta[0]_n(x; q) \\ {}_d\beta[1]_n(x; q) \\ 
      \vdots \\ {}_d\beta[d-1]_n(x; q) \end{bmatrix}
  = (q^{-n})^k \; {}_d\alpha[0]_{n-1}(xq; q)
  \begin{bmatrix} 1 \\ q^{n} \\ \vdots \\ q^{(d-1)n} \end{bmatrix}
\end{equation*}
These matrix equations together with the initial conditions set in Lemma \ref{lemmaSeries}
produce the product identity friendly series ${}_dG^s_{k, a}(x; q)$, 
which reduce to the known series for $d = 1$ and $d = 2$, $s = 0$.  

Then we use Corollary \ref{corSeriesGenFunc} to identify the series as partition generating functions, 
and the equivalence of Corollary \ref{corFuncEq} and Lemma \ref{lemmaRecurrences}
gives us recurrences \eqref{eqRecRegular}, \eqref{eqInitNoPart} and \eqref{eqInitaEqZero}.  
Now, one stipulates that these are induced by conditions $\pmod{d}$ that pertain to 
frequencies of two consecutive parts.  
One already has Gordon's conditions, namely $f_i + f_{i+1} < k$ and $f_1 < a$.  
Also, the recurrence \eqref{eqRecRegular} relates two partitions 
one of which is obtained from the other by deleting 1's and subtracting 1 from the remaining parts.  
It is fairly easy to get the rest of the definition of ${}_db^s_{k, a}(m, n)$, 
and hence, of ${}_dB^s_{k, a}(n)$.  

Same route with overpartitions brings the definition of the $\rho-$statistic
(Definition \ref{defRhoStat}) as well, used in the definition of ${}_d\overline{B}^s_{k, a}(n)$.  
Incidentally, the $\rho-$statistic reduces to the $V-$statistic for overpartitions 
\cite{Chen-etal-overptn-Bressoud, Ctl-etal-overptn-RR-evenmod} in cases $d = 1$ and $d = 2$, 
and one has to be content with the companion identity to Bressoud's theorem for overpartitions.  
However, the point is, although we do not have infinite product representations for $d > 2$, 
there is more to the multiplicity condition side than it meets the eye.  

The bottleneck in the construction is the initial condition \eqref{eqSeriesInitaEqZero}.  
This initial condition alone brings the extra unpleasant conditions 
$2(a+s) \equiv 2(k+1) \equiv 0 \pmod{d}$ or $d = 1$ or $d = 2$
in Theorem \ref{thMain}.  

Thus, a problem for future investigation is: 
Can one find another condition $\pmod{d}$ on the multiplicity side
for overpartitions on top of Gordon's condition 
which reduce to Bressoud type generalization for overpartitions when $d = 2$, 
and still yield congruence conditions on the right hand side for $d \geq 3$? 
In other words, can one have a nicer condition $\pmod{d}$ which brings infinite products, 
or some combination of infinite products as generating functions on the right hand side?  

Rogers-Ramanujan-Gordon theorem and their extensions go hand in hand with 
Andrews-Gordon type identities.  
In other words, some of the partition counters in Definition \ref{defEnumerants} 
have their generating functions expressed as multiple sums.  
Known examples are the Andrews-Gordon identities themselves \cite{Andrews-Gordon}, 
their counterpart for Bressoud's all-moduli generalization \cite{Br-Andrews-Gordon}, 
Andrews-Gordon identities for overpartitions \cite{Chen-etal-RRG-over}, 
and their Bressoud style generalization for overpartitions \cite{Sang-Shi-AndrewsGordon-Overptn}.  
These constitute multiple series generating functions for 
${}_1B^0_{k, a}(n)$, ${}_2B^0_{k, a}(n)$, ${}_1\overline{B}^0_{k, a}(n)$, and ${}_2\overline{B}^0_{k, a}(n)$.  

Using Gordon marking for partitions \cite{K-gordonmarking}, it is possible to have 
a multiple series generating function for ${}_2B^1_{k, a}(n)$, 
and using Gordon marking for overpartitions \cite{Chen-etal-RRG-over}, it must be possible 
to have a multiple series generating function for ${}_2\overline{B}^1_{k, a}(n)$.  
This takes care of all cases when $d = 1$ or $d = 2$.  

Another open question, therefore, 
is the possibility of construction of multiple series generating functions for
${}_dB^s_{k, a}(n)$ and ${}_d\overline{B}^s_{k, a}(n)$ for $d \geq 3$.  
Gordon marking for either (regular) partitions or overpartitions 
behaves nicely with parity, but does not seem to do so with higher moduli.  

\bibliographystyle{amsplain}

\end{document}